\documentclass[12pt,reqno]{amsart}

\usepackage[utf8]{inputenc}
\usepackage{graphicx}
\usepackage{amsmath,amssymb}
\usepackage{hyperref}
\usepackage{xcolor}
\usepackage{bm}
\usepackage{comment}
\usepackage{times}
\usepackage{mathtools}

\numberwithin{equation}{section}

\allowdisplaybreaks

\newtheorem{theorem}{Theorem}[section]
\newtheorem*{theorem*}{Theorem}
\newtheorem{definition}[theorem]{Definition}

\newtheorem{example}[theorem]{Examples}
\newtheorem{remark}[theorem]{Remark}
\newtheorem{lemma}[theorem]{Lemma}

\newtheorem{corollary}[theorem]{Corollary}

\newcommand{\ds}{\displaystyle}
\newcommand{\R}{{\mathbb R}}

\newcommand{\N}{{\mathbb N}}  

\newcommand{\set}[1]{\{ #1 \}}

\newcommand{\ddiff}{\nabla}

\newcommand{\norm}[1]{\lVert #1 \rVert}

\newcommand{\diff}{{\rm d}}

\newcommand{\dist}{{\rm dist}}

\def\dist{{\rm dist}}

\title[Graphs in the sense of John]{Weighted graphs in the sense of John and\\  a global Poincar\'e inequality}

\date{\today}

\author[F. L\'opez-Garc\'\i a]{Fernando L\'opez-Garc\'\i a}
\address{Department of Mathematics and Statistics\\ California State Polytechnic University Pomona\\ Pomona, CA 91768, US}
\email{fal@cpp.edu}

\author[J. Rodriguez]{John Rodriguez}
\address{Department of Mathematics, University of Washington\\ Seattle, WA 98195, US}
\email{jrodri7@uw.edu}

\begin{document}

\begin{abstract}
In this paper, we establish a condition on weighted graphs with finite measure that guarantees the validity of a global Poincar\'e inequality. This condition can be viewed as a discrete analogue of the criterion introduced by J.~Boman in 1982 for Whitney cubes, which in turn characterizes the condition originally proposed by F.~John in his seminal 1961 work.
\end{abstract}



\maketitle

\section{Introduction}

The class of John domains in $\R^n$ is the largest class of domains, under very general assumptions, for which the improved Poincar\'e inequality and the Sobolev-Poincar\'e inequality hold, together with many related results. For details, we refer the reader to \cite{Hs, ADM, BK, Lg3} and the references therein. This class of domains was introduced by Fritz John in 1961 (see \cite{J}) via the so-called twisted cone condition, which asserts that any point $x$ in the domain can be connected to a fixed base point $x_0$ by a path such that, at any point along the path, the distance to the boundary is comparable to the distance traveled from $x$. In 1982, J. Boman provided in \cite{B1} an equivalent formulation in terms of chains of Whitney cubes. Later, the first author of this manuscript introduced in \cite{Lg2} another equivalent definition of John domains, similar to that in \cite{B1}, where chains are replaced by a tree structure. In this paper, we adapt the notion from \cite{Lg2} to weighted graphs. This formulation will be developed in Section \ref{section: Decomposition}; however, for convenience, we include the definition here. Given a connected graph $G=(V,E)$, equipped with a positive vertex function $\mu:V\to\R$, we are interested in graphs for which there exists a rooted spanning tree $(V,E_T)$ of $G$ (i.e. a connected and acyclic subgraph that contains every vertex in $V$ with a distinguished element $a$ as its root) and a constant $c\geq 1$ such that 
\begin{equation}
\label{intro: John graph}
    \mu(S_t) \le c\mu(t) \quad \text{for all } t \in V,
\end{equation}
where $S_t \coloneqq \set{s \in V : s\succeq t}$ is the \emph{shadow} of $t$. In this definition, we say that $s\succeq t$ if the unique path in the spanning tree from $s$ to the root $a$ contains $t$.

Following the Euclidean case, we show that this condition implies a global version of the Poincar\'e inequality. This implication is obtained through a \emph{local-to-global} argument studied in \cite{Lg}, which is based on decomposition of functions. In addition, we provide an estimate for the constant appearing in the Poincar\'e inequality in terms of the constant in \eqref{intro: John graph}. These estimates are of particular interest for bounding the first nonzero eigenvalue of the Laplacian. This relationship is well known in both continuous and discrete settings; see \cite{G,C} and references therein for details.

The notion of derivatives we use here is standard and defined as follows. Given a function $f:V\to \R$, the \emph{length of the gradient} of $f$ is the map $|\ddiff f| : V \to \R_{\ge 0}$ given by
    \begin{equation*}
        |\ddiff f|(t) = \sum_{s : s\sim t} |f(t)-f(s)| \quad \text{for all } t\in V,
    \end{equation*}
where $s\sim t$ indicates that the vertices $s$ and $t$ are adjacent in the graph $G=(V,E)$.

As a slight abuse of notation, we write $(V,\mu)$ for the weighted graph $(V,E,\mu)$.

Our main result is the following:

\begin{theorem*}
    Let $(V,\mu)$ be a weighted graph with bounded degree that satisfies the condition in Definition \ref{def: John domain}. Also, let $p\in [1,\infty)$. Then, there exists a constant $C_P>0$ such that the inequality
    \begin{equation*}
        \norm{f}_{\ell^p(V,\mu)} \le C_P \norm{|\ddiff f|}_{\ell^p(V,\mu)}
    \end{equation*}
    holds for any $f\in \ell^p(V,\mu)$ that sums zero with respect to $\mu$ on $V$. 
\end{theorem*}

The paper is organized as follows. In Section~\ref{section: Hardy}, we study the boundedness of a Hardy-type operator, which plays a central role in our approach. Section~\ref{section: Decomposition} is devoted to a function decomposition that allows us to pass the validity of the Poincar\'e inequality on edges to the entire weighted graph. Finally, in Section~\ref{section: Poincare}, we establish the global Poincar\'e inequality.

\section{Some basic definitions and an averaging operator on Graphs with Rooted Spanning Trees}\label{section: Hardy}

Let us recall some basic and well-known definitions on graphs and introduce an averaging operator that we use later.

A graph $G$ is a pair $(V, E)$, where $V$ is a set of \emph{vertices}, i.e., an arbitrary set whose elements are called vertices, and $E$ is a set of \emph{edges} where the elements of $E$ are unordered pairs $(x, y)$ of vertices $x,\, y \in V$. We write $x \sim y$ if $(x, y) \in E$ and say that $x$ is \emph{adjacent} to $y$. The edge $(x, y)$  is also denoted by $xy$, and $x, y$ are called the \emph{endpoints} of this edge. In this work, we assume that the vertex set $V$ is countable. 

A \emph{chain} in $G$ is a finite sequence of vertices $t_0, t_1, \dots, t_n$ such that $t_i$ is adjacent to $t_{i-1}$ for any $1 \leq i\leq n$. The graph $G=(V,E)$ is \emph{connected} if for any $s,t\in V$ there exists a chain connecting $s$ to $t$. 

A graph $G$ is called \emph{locally finite} if every vertex $x$ has finitely many adjacent vertices. The \emph{degree} of a vertex $x$, denoted $\deg(x)$, is the number of edges incident to $x$. We say that $G$ has \emph{bounded degree} if there exists an integer $M$ such that $\deg(x)\leq M$ for all $x\in V$. A graph $G$ is called \emph{finite} if the number of its vertices is finite.

A graph is a \emph{tree} if it is connected and has no cycles. A \emph{rooted tree} is a tree with a distinguished vertex $a$, called the \emph{root}. If $G$ is a rooted tree, we can define a partial order $\succeq$ on $V$ by declaring that $s \succeq t$ if the unique path from $s$ to the root $a$ contains $t$. We write $s \succ t$ if $s \succeq t$ and $s\neq t$. 

A vertex $k$ is the \emph{parent} of a vertex $t$ if $t \succ k$ and $k$ is adjacent to $t$. For simplicity, we write $k=t_p$. Analogously, $t$ is a \emph{child} of $k$ if $k=t_p$. 

Finally, a \emph{spanning tree} of a connected graph $G=(V,E)$ is a minimal connected subgraph of $G$ that includes all vertices and has no cycles. We use the following notation: 
\begin{equation}
\label{def: spanning tree}
G_T:=(V,E_T),
\end{equation}
where $E_T\subseteq E$ is the edge set of the spanning tree. In this case, we say $x \sim_T y$ if the pair $(x, y) \in E_T$. Naturally, if $x \sim_T y$ then $x \sim y$.

A weighted graph is a graph $G=(V,E)$ provided with a positive function $\mu:V\to\R$. By abuse of notation, we denote weighted graphs by $G=(V,\mu)$. We say that $(V,\mu)$ has \emph{finite measure} if 
\begin{equation*}
\mu(V)=\sum_{s\in V} \mu(s)<\infty. 
\end{equation*}
Furthermore, let $\Omega\subseteq V$. Then a function $f : V \to \R$ \emph{sums zero with respect to $\mu$ on $\Omega$} if 
\begin{equation*}
        \sum_{s\in \Omega} f(s)\mu(s)=0.
\end{equation*}

The space $\ell^q(V, \mu)$, or simply $\ell^q(V)$, for $1 \leq q < \infty$, consists of all functions $f : V \to \R$ such that
\begin{equation*}
\| f \|_q := \left( \sum_{s \in V} |f(s)|^q \mu(s) \right)^{1/q} < \infty.
\end{equation*}
For $q = \infty$, the space $\ell^\infty(V)$ is defined as the set of all bounded functions $f : V \to \mathbb{C}$, with norm
\begin{equation*}
\| f \|_\infty := \sup_{s \in V} |f(s)|.
\end{equation*}
Notice that in weighted graphs with finite measure $\ell^q(V)\subseteq \ell^1(V)$, for any $1 \leq q \leq \infty$.  

Finally, let us define the following averaging operator; also known in this manuscript as a Hardy-type operator.
\begin{definition}
\label{def: Hardy op}
Let $(V,E,\mu)$ be a weighted, connected graph with finite measure and a distinguished spanning tree $(V,E_T)$ that induces a partial order $\succeq$ on the vertices. We define the following \emph{Hardy-type operator} $T$ on $\ell^1(V,\mu)$ 
    \begin{equation*}
        Tf(t) \coloneqq  \frac{1}{\mu(S_t)}\sum_{s:s \succeq t} |f(s)|\mu(s),
    \end{equation*}
for all $t\in V$, where $S_t \coloneqq \set{s \in V : s\succeq t}$ is called the \emph{shadow} of the vertex $t$.
\end{definition}

\begin{remark}
Notice that this averaging operator $T$ depends on the chosen spanning tree.
\end{remark}

The following theorems state that the operator $T$ defined above is strong $(\infty, \infty)$ and weak $(1,1)$ bounded. We then prove the strong $(q,q)$ boundedness of $T$, for $1<q<\infty$, using the Marcinkiewicz interpolation Theorem (see \cite[Theorem 1.3.1]{BL} and \cite{N}.

\begin{theorem}
\label{thm: Hardy strong (infty,infty) cont}
The Hardy-type operator $T$ in Definition \ref{def: Hardy op} is strong $(\infty, \infty)$ bounded, with constant $C=1$. Namely, 
\begin{equation*}
    \norm{Tf}_{\ell^\infty(V)} \le \norm{f}_{\ell^\infty(V)},
\end{equation*}
for any $f\in \ell^\infty(V)$.
\end{theorem}

\begin{proof}
This result follows from the fact that $T$ is an averaging operator. Indeed, given a function $f\in\ell^\infty(V)$, we acquire the following
    \begin{align*}
        \norm{Tf}_{\ell^\infty(V)} &= \sup_{t \in V} |Tf(t)| \\ 
        &= \sup_{t \in V} \frac{1}{\mu(S_t)} \sum_{s \succeq t} |f(s)|\mu(s) \\
        &\le \norm{f}_{\ell^\infty(V)} \sup_{t \in V} \frac{1}{\mu(S_t)}\sum_{s \succeq t} \mu(s) \\
        &= \norm{f}_{\ell^\infty(V)} \sup_{t \in V} \frac{\mu(S_t)}{\mu(S_t)} \\
        &= \norm{f}_{\ell^\infty(V)}.
    \end{align*}
    Therefore, $T$ is bounded with $\norm{T}_{\ell^\infty(V)\to\ell^\infty(V)}\le 1$, and thus strong $(\infty, \infty)$ bounded.
\end{proof}

We proceed by proving the following lemma, which is used in the argument to conclude the weak $(1,1)$ boundedness of $T$. 

\begin{lemma}
    \label{lem: Hardy weak (1,1) cont}
    Let $G=(V,E)$ be a graph with a rooted spanning tree $(V,E_T)$ that induces on $V$ the partial order $\succeq$.  Let $t_1, t_2 \in V$. If $S_{t_1} \cap S_{t_2} \ne \emptyset$, then either $S_{t_1} \subseteq S_{t_2}$ or $S_{t_2} \subseteq S_{t_1}$. 
    
\end{lemma}

\begin{proof}
    Since $S_{t_1} \cap S_{t_2} \ne \emptyset$, there exists $t \in S_{t_1} \cap S_{t_2}$. It follows that $t \succeq t_1$ and $t \succeq t_2$. Now, let us assume, to the contrary, that $t_1 \not\succeq t_2$ and $t_2 \not\succeq t_1$. Then the path from $t$ through $t_1$ to the root $a$ is distinct from the path from $t$ through $t_2$ to the root $a$. The existence of these two distinct paths from $t$ to the root $a$ implies the existence of a cycle, which contradicts the fact that the partial order is defined by a spanning tree. Therefore, either $t_1 \succeq t_2$ or $t_2 \succeq t_1$ which gives that $t_1 \in S_{t_2}$ or $t_2 \in S_{t_1}$. It follows that $S_{t_1}\subseteq S_{t_2}$ or $S_{t_2}\subseteq S_{t_1}$.
\end{proof}

Now, we prove the weak $(1,1)$ boundedness of $T$. 

\begin{theorem}
\label{thm: Hardy weak (1,1) cont}
    The Hardy-type operator $T$ in Definition \ref{def: Hardy op} is weak $(1,1)$ bounded with constant $C=1$. Namely, for all $f\in \ell^1(V,\mu)$ and $\lambda>0$,
    \begin{equation*}
        \mu(V_{Tf}(\lambda))< \frac{1}{\lambda}\norm{f}_{\ell^1(V,\mu)},
    \end{equation*}
    where the subset of vertices $V_{Tf}(\lambda) \coloneqq \set{t \in V : |Tf(t)| > \lambda}$.
\end{theorem}

\begin{proof}
    Let us define the subset of \emph{minimal vertices} $M_{Tf}(\lambda) \subseteq V_{Tf}(\lambda)$ as 
    \begin{equation*}
        M_{Tf}(\lambda) \coloneqq \set{t \in V_{Tf}(\lambda) : s \not\in V_{Tf}(\lambda) \text{ for all } s \prec t}.
    \end{equation*}
    It follows from its definition that 
    \[V_{Tf}(\lambda) \subseteq \bigcup_{t\in M_{Tf}(\lambda)} S_t.\] Consequently,
    \begin{equation*}
        \mu(V_{Tf}(\lambda)) \le \mu\left(\bigcup_{t\in M_{Tf}(\lambda)} S_t\right)
        = \sum_{t \in M_{Tf}(\lambda)} \mu(S_t).
    \end{equation*}
    The last identity is justified by Lemma \ref{lem: Hardy weak (1,1) cont} and by the definition of the set of minimal vertices $M_{Tf}(\lambda)$. More specifically, no two shadows of vertices in $M_{Tf}(\lambda)$ intersect.
    
Now, since $M_{Tf}(\lambda)\subseteq V_{Tf}(\lambda)$, we find that
    \begin{align*}
        \mu(V_{Tf}(\lambda)) &\le \sum_{t\in M_{Tf}(\lambda)} \mu(S_t)\\
        &< \sum_{t \in M_{Tf}(\lambda)} \mu(S_t)\frac{Tf(t)}{\lambda} \\
        &= \frac{1}{\lambda}\sum_{t \in M_{Tf}(\lambda)} \sum_{s \succeq t} |f(s)|\mu(s) \\
        &\le \frac{1}{\lambda}\norm{f}_{\ell^1(V,\mu)}.
    \end{align*}
     We conclude that $T$ is weak $(1,1)$ bounded.
\end{proof}

\begin{theorem}
\label{thm:Hardy strong (p,p) cont}
    The Hardy-type operator $T$ in Definition \ref{def: Hardy op} is strong $(q,q)$ bounded for $1 < q < \infty$. Moreover, for all $f\in \ell^q(V,\mu)$,
    \begin{equation*}
        \norm{Tf}_{\ell^q(V,\mu)}\le \left(\frac{2^q q}{q-1}\right)^{1/q}\norm{f}_{\ell^q(V,\mu)}.
    \end{equation*}
\end{theorem}

\begin{proof}
The proof of this result follows from Theorem~\ref{thm: Hardy strong (infty,infty) cont} and Theorem~\ref{thm: Hardy weak (1,1) cont}, combined with the Marcinkiewicz interpolation theorem in the discrete setting (see Theorem~1.3.1 in \cite{BL} for a reference on the Marcinkiewicz interpolation theorem).
\end{proof}
The boundedness of some related discrete Hardy-type operators on trees has been studied in the literature; see, for instance, \cite{LgO} and \cite{ACLM}. The proof of Theorem 2.5 relies on a classical geometric argument in tree theory, which
is, in some sense, generalized in \cite[Theorem 6]{ACLM}. A similar argument is also used in \cite[Lemma 3.1]{Lg}.

%%%%%%%%%%%%%%%%%%%%%%%%%%%%%%%%%%%%%%%%%%%%%%%%%%%%%%%%%%%%%%%%%%%%%%%%%%%%%

%%%%%%%%%%%%%%%%%%%%%%%%%%%%%%%%%%%%%%%%%%%%%%%%%%%%%%%%%%%%%%%%%%%%%%%%%%%%%

\section{A Decomposition of Functions Technique}\label{section: Decomposition}

In this work, we use a local-to-global argument to prove the validity of a certain Poincar\'e-type inequality on weighted graphs, relying on the fact that such inequalities are valid on segments or edges. This technique requires writing functions with vanishing mean value on the entire graph as sums of functions supported on segments that also have vanishing mean value. The \emph{boundedness} of this decomposition depends on the \emph{geometric} condition on the graph introduced in the following definition. 

For the remainder of the paper, we assume that the weighted graph $(V,E,\mu)$ has finite measure and degree bounded by $M$.

\begin{definition}
\label{def: John domain}
    Let $G = (V,\mu)$ be a weighted and connected graph. We say that $G$ is a graph in the sense of John if there exists a rooted spanning tree and a positive constant $c$ such that
\begin{equation}
\label{eq: John constant}
    \mu(S_t) \le c\mu(t) \quad \text{for all } t \in V,
\end{equation}
where $S_t \coloneqq \set{s \in V : s\succeq t}$ is the \emph{shadow} of $t$.
\end{definition}

This geometric definition for graphs is inspired by the equivalent definition of John domains proved in \cite{Lg2} on Euclidean domains. 

It is worth noting that condition \eqref{eq: John constant} implies that $(V,\mu)$ has finite measure, since
\begin{equation*}
    \mu(V)=\mu(S_a)\le c\mu(a)<\infty,
\end{equation*}
where $a$ denotes the root of the spanning tree. 

Also, note that any finite weighted graph $(V,\mu)$ satisfies \eqref{eq: John constant} with a constant \[c = \dfrac{\mu(V)}{\min_{t\in V} \mu(t)}.\]

Next, we present three weighted graphs: the first two satisfy \eqref{eq: John constant}, whereas the last one does not.

\begin{example}
\label{ex: k-ary tree is john}
    Let $(V,\mu)$ be a weighted rooted $k$-ary tree (i.e. a tree where each vertex has $k$ children), where the weight is defined by $\mu(t) = \alpha^{\dist(t,a)}$, for every $t\in V$, for some $0<\alpha<1/k$. Here, $dist(t,a)$ denotes the number of edges in the chain that connects $t$ to the root $a$. Then $(V,\mu)$ is John. Indeed, for any $t\in V$, we have the following estimate:
    \begin{equation*}
        \frac{\mu(S_t)}{\mu(t)} = \frac{\sum_{s\in S_t}\mu(s)}{\mu(t)} = \frac{\sum_{s\succeq t}\alpha^{\dist(s,a)}}{\alpha^{\dist(t,a)}} = \sum_{s\succeq t} \alpha^{\dist(s,a) - \dist(t,a)} = \sum_{s\succeq t} \alpha^{\dist(s,t)}
    \end{equation*}
    where the last line holds since $s\succeq t$ and therefore contains $t$ on its unique path to the root $a$. Now, $(V,\mu)$ is a $k$-ary tree, so each vertex has $k$ children. Namely, we can continue the estimation in the following way:
    \begin{equation*}
        \frac{\mu(S_t)}{\mu(t)} = \sum_{s\succeq t} \alpha^{\dist(s,t)} = \sum_{n\ge 0} k^n\alpha^n = \sum_{n \ge 0} (k\alpha)^n = \frac{1}{1-k\alpha}.
    \end{equation*}
 So $k$-ary trees, with the weight introduced above, satisfy condition \eqref{eq: John constant} with a constant estimated by $1/(1-k\alpha)$. 
\end{example}

\begin{example}
\label{ex: John domains}
Given a bounded domain $\Omega \subset \mathbb{R}^n$ and a Whitney decomposition $\{Q_t\}_{t \in V}$, we define the graph $G=(V,E)$, where the vertices correspond to the Whitney cubes, and two vertices $t$ and $s$ are adjacent in $G$ if and only if $Q_t$ and $Q_s$ intersect along an $(n-1)$-dimensional face of one of them. We refer the reader to \cite{S} for details on the existence and properties of Whitney cubes. Furthermore, we define the weight $\mu: V\to \R$ as the Lebesgue measure of each Whitney cube (i.e. $\mu(t):=|Q_t|$). 

Now, it was shown in \cite{Lg2} that if $\Omega$ is a John domain in the sense of Fritz John there exists a Whitney decomposition and a spanning tree $(V,E_T)$ of the graph of Whitney cubes such that 
\begin{equation*}
    Q_s \subseteq K Q_t
\end{equation*}
for any $s, t \in V$, with $s \succeq t$, where $K Q_t$ is a $K$-dilation of $Q_t$. Hence, using that Whitney cubes have disjoint interiors, we conclude that 
\begin{align*}
    \mu(S_t)&=\sum_{s \succeq t} \mu(s) = \sum_{s \succeq t} |Q_s|\\  &=\left| \bigcup_{s \succeq t} Q_s\right| \leq |K Q_t|=K^n |Q_t|=K^n \mu(t).
\end{align*}
Therefore, this graph satisfies the condition in Definition \ref{def: John domain} with constant $c=K^n$.
\end{example}

As we mentioned, every weighted graph that satisfies the condition in Definition \ref{def: John domain} has finite measure. However, not every weighted graph of finite measure satisfies condition \eqref{eq: John constant}. Let us show an example of this fact. 

\begin{example}
\label{ex: finite measure but not John}
    Let $(V,\mu)$ be the weighted path graph on $V=\N_{\ge 2}$, where adjacency is given by consecutive integers. The weight is defined by $\mu(n) = n^{-1}\ln(n)^{-\gamma}$ for all $n \in \N_{\ge 2}$, where $\gamma > 1$. 

Consider the function $f(x) = x^{-1}\ln(x)^{-\gamma}$. Observe that $f(x)$ is continuous, positive, and decreasing in $[2, \infty)$. Since this holds, the integral test implies that
    \begin{equation*}
        \mu(V) = \mu(\N_{\ge 2}) = \sum_{n \ge 2} \mu(n) = \sum_{n \ge 2} \frac{1}{n(\ln(n))^\gamma} < \infty.
    \end{equation*}
    So $(V,\mu)$ has a finite measure. Now, let $n \in \N_{\ge 2}$ and observe that
    \begin{equation*}
        \frac{\mu(S_n)}{\mu(n)} = \frac{\sum_{k \ge n} \frac{1}{k(\ln(k))^\gamma}}{\frac{1}{n(\ln(n))^\gamma}} = n(\ln(n))^\gamma \sum_{k \ge n} \frac{1}{k(\ln(k))^\gamma}.
    \end{equation*}
    The integral of the function $f(x)$ previously mentioned provides a lower estimate for the sum above in the following way:
    \begin{align*}
        \frac{\mu(S_n)}{\mu(n)} &= n(\ln(n))^\gamma \sum_{k \ge n} \frac{1}{k(\ln(k))^\gamma}\\
        &\ge n(\ln(n))^\gamma \int_n^\infty \frac{1}{x(\ln(x))^\gamma}\,\diff x\\
        &= n(\ln(n))^\gamma \frac{1}{(\gamma - 1)(\ln(n))^{\gamma-1}}\\
        &= \frac{n\ln(n)}{\gamma - 1}.
    \end{align*}
    Hence, $\mu(S_n)/\mu(n)$ is not uniformly bounded and $(V,\mu)$ does not satisfy \eqref{eq: John constant}.
\end{example}

Now, let us define the decomposition of functions that we use in this local-to-global argument to prove the $\ell^p(V,\mu)$-Poincar\'{e} inequality for $1 \le p < \infty$. 

\begin{definition}
\label{def: decomp of f}
    Let $(V,\mu)$ be a weighted graph equipped with a rooted spanning tree that defines a partial order $\succeq$. Given $f\in \ell^1(V,\mu)$, we introduce the collection of functions $\set{f_t}_{t\in V^*}$, where $V^* = V\setminus\set{a}$. For each $t\in V^*$, the function $f_t : V \to \R$ is defined by
    \begin{equation}
    \label{eq: def of decomp}
        f_t(s):=
        \begin{cases}
            0 & s\ne t \text{ and } s\ne t_p,\\
            \ds f(t) + \frac{1}{\mu(t)}\sum_{k\succ t} f(k)\mu(k) & s=t,\\
            \ds-\frac{1}{\mu(t_p)}\sum_{k\succeq t} f(k)\mu(k) & s=t_p,
        \end{cases}
    \end{equation}
    for any $s\in V$. 
\end{definition}

Observe that this collection of functions is well defined since $f\in \ell^1(V,\mu)$. The following results state some other properties that the functions in \eqref{eq: def of decomp} satisfy. 

\begin{theorem}
\label{thm: decomp of f}
    Let $(V,\mu)$ be a weighted graph satisfying the assumptions stated in Definition \ref{def: decomp of f}. Now, let $f\in\ell^1(V,\mu)$, a function with sum zero with respect to $\mu$ on $V$. Then the functions $f_t$ in the collection $\set{f_t}_{t\in V^*}$ presented in (\ref{eq: def of decomp}) describe a decomposition of $f$ supported on the segments $\set{t,t_p}$ such that $f=\sum_{t\in V^*}f_t$ and $f_t$ sums zero with respect to $\mu$ on $V$ for every $t\in V^*$.
\end{theorem}

\begin{proof}
    Notice that $f_t$ is indeed supported on the segment $\set{t,t_p}$ for each $t\in V^*$. Now, let $s$ be a vertex other than the root. Then, by using (\ref{eq: def of decomp}), we attain the following
    \begin{align*}
        \sum_{t\in V^*} f_t(s) &= \sum_{t : s=t \text{ or } s=t_p} f_t(s) \\
        &= f_s(s) + \sum_{t : s=t_p} f_t(s) \\
        &= f(s) + \frac{1}{\mu(s)}\sum_{k\succ s}f(k)\mu(k) + \sum_{t : s=t_p}\left(-\frac{1}{\mu(s)}\sum_{k\succeq t}f(k)\mu(k)\right).
    \end{align*}
    Furthermore, the sum over all vertices $t$ such that $s=t_p$ converges since the shadow of each vertex $t$ is disjoint and $f\in \ell^1(V,\mu)$. We now derive the following
    \begin{align*}
        \sum_{t\in V^*} f_t(s) &= f(s) + \frac{1}{\mu(s)}\sum_{k\succ s}f(k)\mu(k) -\frac{1}{\mu(s)}\sum_{t:s=t_p}\sum_{k\succeq t}f(k)\mu(k)\\
        &=f(s) + \frac{1}{\mu(s)}\sum_{k\succ s}f(k)\mu(k) -\frac{1}{\mu(s)}\sum_{k\succ s}f(k)\mu(k)\\
        &= f(s).
    \end{align*}
    If $s$ is the root, that is $s=a$, then $f_t(a) = 0$ for each $t$ that is not a child of the root $a$. Additionally,
    \begin{align*}
        \sum_{t\in V^*} f_t(s) &= \sum_{t : a=t_p} f_t(a) \\
        &= \sum_{t : a=t_p} \left(-\frac{1}{\mu(a)}\sum_{k\succeq t}f(k)\mu(k)\right)\\
        &= -\frac{1}{\mu(a)}\sum_{t:a=t_p}\sum_{k\succeq t}f(k)\mu(k)\\
        &= -\frac{1}{\mu(a)}\sum_{k\succ a}f(k)\mu(k)\\
        &=f(a),
    \end{align*}
    where the last line holds since $f$ has sum zero with respect to $\mu$ on $V$. Whence, we conclude that $f$ can be decomposed by $f_t$ for $t\in V^*$. We now aim to show that each $f_t$ sums zero with respect to $\mu$ on $V$. To that end, let $t\in V^*$. Then
    \begin{align*}
        \sum_{s\in V} f_t(s)\mu(s) &= \sum_{s : s=t \text{ or } s=t_p} f_t(s)\mu(s)\\
        &= \left(f(t) + \frac{1}{\mu(t)}\sum_{k\succ t} f(k)\mu(k)\right)\mu(t) - \left(\frac{1}{\mu(t_p)}\sum_{k\succeq t} f(k)\mu(k)\right)\mu(t_p)\\
        &= \sum_{k\succeq t} f(k)\mu(k) - \sum_{k\succeq t} f(k)\mu(k) = 0,
    \end{align*}
    as desired.
\end{proof}

Notice that the previous proof does not require the validity of condition \eqref{eq: John constant}; however, this condition is used in the proof of the following result. The next corollary follows from the boundedness of the Hardy-type operator $T$.

\begin{corollary}
\label{cor: sum||f_t||<=K||f||}
    Let $(V,\mu)$ be a weighted graph with degree bounded by $M$ that satisfies \eqref{eq: John constant}. Then, for every $g\in \ell^q(V,\mu)$, with $q\in (1,\infty)$, the collection of functions $\set{g_t}_{t\in V^*}$ defined in (\ref{eq: def of decomp}) satisfies the following estimate:
    \begin{equation}
    \label{eq: sum||f_t||<=K||f||}
        \sum_{t\in V^*}\norm{g_t}_{\ell^q(V,\mu)}^q\le c^q M \frac{2^q q}{q-1} \norm{g}_{\ell^q(V,\mu)}^q, 
    \end{equation}
    where $c$ is the constant that appears in Definition \ref{def: John domain}.
\end{corollary}

\begin{proof}
    We begin our estimation argument as follows:
    \begin{equation}   
    \label{eq: sum||f_t|| only on s=t and s=t_p}
        \sum_{t\in V^*} \norm{g_t}_{\ell^q(V,\mu)}^q = \sum_{t\in V^*} |g_t(t)|^q\mu(t)+|g_t(t_p)|^q\mu(t_p),
    \end{equation}
    where we wish to estimate both $g_t(t)$ and $g_t(t_p)$ by the Hardy-type operator in Definition \ref{def: Hardy op}. The strong $(q,q)$ boundedness of $T$ implies the final estimation. Thus, we   proceed as follows:
    \begin{equation*}
        |g_t(t)| = \left|g(t) + \frac{1}{\mu(t)}\sum_{k\succ t} g(k)\mu(k)\right|=\left|\frac{1}{\mu(t)}\sum_{k\succeq t} g(k)\mu(k)\right|\leq \frac{\mu(S_t)}{\mu(t)}\frac{1}{\mu(S_t)}\sum_{k\succeq t} |g(k)|\mu(k).
    \end{equation*}
    Then $|g_t(t)|\le cTg(t)$, where $c$ is the constant in Definition \ref{def: John domain}. Similarly, we give a bound for $|g_t(t_p)|$:
    \begin{equation*}
        |g_t(t_p)| = \left|-\frac{1}{\mu(t_p)}\sum_{k\succeq t} g(k)\mu(k)\right| \le \frac{\mu(S_{t_p})}{\mu(t_p)} \frac{1}{\mu(S_{t_p})}\sum_{k\succeq {t_p}} |g(k)|\mu(k).
    \end{equation*}
    Hence, $|g_t(t_p)|\le c Tg(t_p)$. Now, observe that Theorem \ref{thm:Hardy strong (p,p) cont} and the fact that $g\in \ell^q(V,\mu)$ implies that $Tg\in \ell^q(V,\mu)$. Hence, from (\ref{eq: sum||f_t|| only on s=t and s=t_p}) we can conclude that
    \begin{align*}
        \sum_{t\in V^*} \norm{g_t}^q_{\ell^q(V,\mu)} &\le c^q \sum_{t\in V^*} |Tg(t)|^q\mu(t) + c^q \sum_{t\in V^*}|Tg(t_p)|^q \mu(t_p)\\ 
        &= c^q \sum_{s\in V^*} |Tg(s)|^q\mu(s) + c^q \sum_{s\in V}|Tg(s)|^q \mu(s) \#\{t\in V^* \colon t_p=s\}\\ 
%        &= c^q \sum_{s\in V} |Tg(s)|^q \mu(s) \#\{t\in V \colon t\sim s\}\\
        &\leq c^q M \sum_{s\in V} |Tg(s)|^q \mu(s)\\
        &\le c^q M \frac{2^q q}{q-1} \norm{g}_{\ell^q(V,\mu)}^q,
    \end{align*}
    which concludes the proof.
\end{proof}

%%%%%%%%%%%%%%%%%%%%%%%%%%%%%%%%%%%%%%%%%%%%%%%%%%%%%%%%%%%%%%%%%%%%%%%%%%%%%

\section{A global weighted Poincar\'e Inequality on Graphs with Rooted Spanning Trees}\label{section: Poincare}

In this section, we prove the validity of a global $\ell^p(V,\mu)$-Poincar\'e inequality on weighted graphs $(V,\mu)$, that satisfy the condition in Definition \ref{def: John domain}, via a local-to-global argument. This technique requires us to prove that the $\ell^p(V,\mu)$-Poincar\'e inequality is satisfied locally. In particular, we prove that it is satisfied on any edge of a general weighted graph $(V,\mu)$. We prove this fact since the collection of decomposed functions $\set{g_t}_{t\in V^*}$ in Definition \ref{def: decomp of f} is supported exactly on the segments $\set{t,t_p}\subseteq V$ for a spanning tree in the graph. Through a dual argument, we are able to utilize the decomposition and thus the validity of the local Poincar\'e inequality to prove the global Poincar\'e inequality.

Before presenting the proofs of the local and global $\ell^p(V,\mu)$-Poincar\'e inequalities, we recall the notion of gradient that we use in this work.

\begin{definition}
\label{def: length(grad f)}
    Given a locally finite graph $G = (V,E)$ and a function $f : V \to \R$, we define the \emph{length of the gradient} of $f$ as the function $|\ddiff f| : V \to \R_{\ge 0}$ defined by
    \begin{equation}
    \label{eq: length(grad f)}
        |\ddiff f|(t) = \sum_{s : s\sim t} |f(s)-f(t)| \quad \text{for all } t\in V.
    \end{equation}
\end{definition}

This is a standard notion in analysis on graphs and plays the role of gradients in analysis on metric spaces \cite{LST}. Now, we denote the average of a function $f : V \to \R$ on a subset of vertices $\Omega \subseteq V$ by 
\begin{equation*}
    f_\Omega = \frac{1}{\mu(\Omega)}\sum_{t\in \Omega} f(t)\mu(t).
\end{equation*}

We begin our discussion by proving that a $\ell^p(V,\mu)$-Poincar\'e inequality is satisfied locally on the segments $\set{t, t_p}\subseteq V$.

\begin{lemma}
\label{lem: local poincare}
    Let $G = (V,\mu)$ be a weighted graph with a distinguished spanning tree that defines on $V$ a partial order $\preceq$. Now, given a function $f\in \ell^p(V,\mu)$, $p\in [1,\infty)$, that has zero sum with respect to $\mu$ on the segment $\set{t, t_p} \subseteq V$ for a certain $t\in V$, it follows that 
    \begin{equation*}
        \norm{f}_{\ell^p(\set{t,t_p}, \mu)}\le \norm{|\ddiff_t f|}_{\ell^p(\set{t,t_p}, \mu)}.
    \end{equation*}
Here $|\ddiff_t f|$ denotes the gradient restricted to the subgraph $(\{t,t_p\},\mu)$ of $G$, and takes the value $|f(t)-f(t_p)|$ at both $t$ and $t_p$.
\end{lemma}

\begin{proof}
    Since $f$ has zero sum with respect to $\mu$ on $\set{t,t_p}$, we have $f(t)\mu(t)+f(t_p)\mu(t_p)=0$. Now the weight function $\mu$ is positive, so without loss of generality, $f(t)> 0$ and $f(t_p)<0$. This implies that $|f(t)|\le |f(t)-f(t_p)|$ and $|f(t_p)|\le |f(t)-f(t_p)|$. Thus,
    \begin{align*}
        \norm{f}_{\ell^p(\set{t,t_p},\mu)}^p&= |f(t)|^p\mu(t)+|f(t_p)|^p\mu(t_p)\\
        &\le |f(t)-f(t_p)|^p\mu(t)+|f(t)-f(t_p)|^p\mu(t_p)\\
        &= \left(|\ddiff_t f|(t)\right)^p\mu(t) + \left(|\ddiff_t f|(t_p)\right)^p\mu(t_p)\\
        &= \norm{|\ddiff_t f |}_{\ell^p(\set{t,t_p},\mu)}^p
    \end{align*}
    as desired.
\end{proof}

\begin{theorem}
\label{thm: Poincare}
    Let $(V,\mu)$ be a weighted graph with degree bounded by $M$ that satisfies the condition in Definition \ref{def: John domain}. Also, let $p\in [1,\infty)$. Then, there exists a constant $C_P>0$ such that the inequality
    \begin{equation}
    \label{eq: global poincare}
        \norm{f}_{\ell^p(V,\mu)} \le C_P \norm{|\ddiff f|}_{\ell^p(V,\mu)}
    \end{equation}
    holds for any $f\in \ell^p(V,\mu)$ that has zero sum with respect to $\mu$ on $V$. Furthermore, the constant in the Poincar\'e inequality is upper bounded by a multiple of the geometric constant in \eqref{eq: John constant}. Indeed, 
    \begin{equation*}
C_P\le 2c\, (Mp)^{1-1/p},
    \end{equation*}
where $c$ is the constant that appears in Definition \ref{def: John domain}.
\end{theorem}

\begin{proof}
    First, let us consider the case $p> 1$. We estimate $\norm{f}_{\ell^p(V,\mu)}$ by a dual argument. That is, we have 
    \begin{equation*}
        \norm{f}_{\ell^p(V,\mu)} = \sup_{\norm{g}_{\ell^q(V,\mu)}\le 1} \sum_{s\in V} f(s)g(s)\mu(s),
    \end{equation*}
    where $1/p+1/q=1$. 
    
    Furthermore, the fact that $f$ sums zero with respect to $\mu$ on $V$ implies that
    \begin{equation*}
        \sum_{s\in V} f(s)(g(s) - g_V)\mu(s) = \sum_{s\in V}f(s)g(s)\mu(s) - g_V\sum_{s\in V}f(s)\mu(s) = \sum_{s\in V}f(s)g(s)\mu(s).
    \end{equation*}
    Thus,
    \begin{equation*}
        \norm{f}_{\ell^p(V,\mu)} = \sup_{\norm{g}_{\ell^q(V,\mu)}\le 1} \sum_{s\in V} f(s)(g(s)-g_V)\mu(s),
    \end{equation*}
    where $g-g_V$ sums zero with respect to $\mu$ on $V$. Moreover, notice that 
\begin{align*}
        \norm{g-g_V}_{\ell^q(V,\mu)} &\le \norm{g}_{\ell^q(V,\mu)}+\norm{g_V}_{\ell^q(V,\mu)}=\norm{g}_{\ell^q(V,\mu)}+|g_V|\mu(V)^{1/q} \\[0.85em] 
        &=   \norm{g}_{\ell^q(V,\mu)}+\mu(V)^{1/q-1}\sum_{s\in V}|g(s)|\mu(s)\\
        &\le \norm{g}_{\ell^q(V,\mu)}+\mu(V)^{1/q-1}\mu(V)^{1/p}\norm{g}_{\ell^q(V,\mu)}\le 2\norm{g}_{\ell^q(V,\mu)}\le 2.
    \end{align*}
Therefore, using Theorem \ref{thm: decomp of f}, we can decompose $g-g_V$ according to (\ref{eq: def of decomp}). For simplicity, we write the decomposed functions as $g_t$ for $t\in V^*$. Thus, 
    \begin{equation}
    \label{eq: dual argument sum(g_t)}
        \norm{f}_{\ell^p(V,\mu)} = \sup_{\norm{g}_{\ell^q(V,\mu)}\le 1} \sum_{s\in V} f(s) \left(\sum_{t\in V^*} g_t(s)\right)\mu(s) = \sup_{\norm{g}_{\ell^q(V,\mu)}\le 1} \sum_{s\in V} \sum_{t\in V^*}f(s)g_t(s)\mu(s).
    \end{equation}

Notice that the double sum stated above is absolutely convergent. In fact, we use the support of each function $g_t$ and \eqref{eq: sum||f_t||<=K||f||}: 
\begin{align*}
\sum_{s\in V} \sum_{t\in V^*} |f(s)| |g_t(s)|\mu(s) 
&=\sum_{t\in V^*} \sum_{\substack{s : s=t \text{ or }\\ s=t_p}} |f(s)||g_t(s)|\mu(s)  \\ 
&\le \left(\sum_{t\in V^*} \sum_{\substack{s : s=t \text{ or }\\ s=t_p}} |f(s)|^p \mu(s)\right)^{1/p}
 \left(\sum_{t\in V^*} \sum_{\substack{s : s=t \text{ or }\\ s=t_p}} |g_t(s)|^q
 \mu(s)\right)^{1/q}  \\ 
&\leq M^{1/p}\norm{f}_{\ell^p(V,\mu)}  \left(\sum_{t\in V^*} \norm{g_t}^q_{\ell^q(V,\mu)} \right)^{1/q}.
\end{align*} 
Observe in the middle line that each term $|f(s)|^p\mu(s)$, with $s\in V$ such that $s=t$ or $s=t_p$ for some $t\in V^*$, appears multiple times. Actually, the number of repetitions equals the number of children plus its parent if there is one. Therefore, the double sum in (\ref{eq: dual argument sum(g_t)}) converges absolutely, and we may interchange the order of summation. Now, note that $f_{\set{t,t_p}}$ is the average of $f$ on the segment $\set{t,t_p}$ and so $f - f_{\set{t, t_p}}$ sums zero with respect to $\mu$ on the segment $\set{t,t_p}$. With these notes, we proceed from (\ref{eq: dual argument sum(g_t)}) as follows
    \begin{align}
    \label{eq: time for local poincare}
        \norm{f}_{\ell^p(V,\mu)} &= \sup_{\norm{g}_{\ell^q(V,\mu)}\le 1} \sum_{t\in V^*} \sum_{s\in V}f(s)g_t(s)\mu(s) \nonumber \\
        &= \sup_{\norm{g}_{\ell^q(V,\mu)}\le 1} \sum_{t\in V^*} \sum_{\substack{s : s=t \text{ or }\\ s=t_p}} f(s)g_t(s)\mu(s) \nonumber \\
        &= \sup_{\norm{g}_{\ell^q(V,\mu)}\le 1} \sum_{t\in V^*} \sum_{\substack{s : s=t \text{ or }\\ s=t_p}} (f(s)-f_{\set{t,t_p}})g_t(s)\mu(s)\nonumber \\
        &\le \sup_{\norm{g}_{\ell^q(V,\mu)}\le 1} \sum_{t\in V^*}\norm{f-f_{\set{t,t_p}}}_{\ell^p(\set{t,t_p},\mu)}\norm{g_t}_{\ell^q(\set{t,t_p},\mu)}.
    \end{align}
    The third equality holds because $g_t$ has zero sum with respect to $\mu$ on $V$ by Theorem \ref{thm: decomp of f}; hence, $g_t$ also has zero sum with respect to $\mu$ on the segment $\set{t,t_p}$. The last line follows from H\"{o}lder's inequality. Now, using Lemma \ref{lem: local poincare}, we obtain
    \begin{align*}
        |\ddiff_t(f-f_{\set{t,t_p}})|(s) &= |(f(t)-f_{\set{t,t_p}})-(f(t_p)-f_{\set{t,t_p}})| \\ 
        &= |f(t)-f(t_p)|=|\ddiff_t f|(s),
    \end{align*}
    for $s=t$ or $s=t_p$. Henceforth, using the local Poincar\'e inequality in Lemma \ref{lem: local poincare}, H\"{o}lder's inequality and Corollary \ref{cor: sum||f_t||<=K||f||}, we can conclude that 
    \begin{align*}
        \norm{f}_{\ell^p(V,\mu)} &\le \sup_{\norm{g}_{\ell^q(V,\mu)}\le 1} \sum_{t\in V^*}\norm{|\ddiff_t f|}_{\ell^p(\set{t,t_p},\mu)}\norm{g_t}_{\ell^q(\set{t,t_p},\mu)} \notag\\ 
        &\le \sup_{\norm{g}_{\ell^q(V,\mu)}\le 1} \left(\sum_{t\in V^*}\norm{|\ddiff_t f|}_{\ell^p(\set{t,t_p},\mu)}^p\right)^{1/p}\left(\sum_{t\in V^*}\norm{g_t}_{\ell^q(\set{t,t_p},\mu)}^q\right)^{1/q} \notag\\ 
        &\le c M^{1/q}\left(\frac{2^q q}{q-1}\right)^{1/q}\left(\sum_{t\in V^*}\norm{|\ddiff_t f|}_{\ell^p(\set{t,t_p},\mu)}^p\right)^{1/p}\\ 
        &= c M^{1/q}\left(\frac{2^q q}{q-1}\right)^{1/q}\left(\sum_{t\in V^*}|f(t)-f(t_p)|^p\mu(t)+|f(t)-f(t_p)|^p\mu(t_p)\right)^{1/p} \\  
    &= c M^{1/q}\left(\frac{2^qq}{q-1}\right)^{1/q}\left(\sum_{k\in V} \sum_{s : s\sim_T k} |f(s)-f(k)|^p\mu(k)\right)^{1/p}\\ 
    &\le c M^{1/q}\left(\frac{2^qq}{q-1}\right)^{1/q}\left(\sum_{k\in V} \left(\sum_{s : s\sim_T k} |f(s)-f(k)|\right)^p \mu(k)\right)^{1/p}  \\ 
    &\leq  c M^{1/q}\left(\frac{2^qq}{q-1}\right)^{1/q}\left(\sum_{k\in V} \left(|\ddiff f|(k)\right)^p \mu(k)\right)^{1/p}. 
    \end{align*}
Recall that $s\sim_T k$ means that $s$ and $k$ are adjacent in the spanning tree $(V,E_T)$ that appears in Definition \ref{def: spanning tree}. The last estimate follows from the fact that $s\sim_T k$ implies $s\sim k$. Actually, we prove a slightly stronger result where the length of the gradient in the original graph $(V,E)$ is replaced by that on the spanning tree $(V,E_T)$, which is smaller since $E_T\subseteq E$. Hence, this completes the proof for $1<p<\infty$.

Finally, consider the case $p=1$. We begin as we did in the case $p>1$ and estimate $\norm{f}_{\ell^1(V,\mu)}$ by a dual argument. That is, consider the following
    \begin{equation*}
        \norm{f}_{\ell^1(V,\mu)} = \sup_{\norm{g}_{\ell^\infty(V)}\le 1} \sum_{s\in V} f(s)g(s)\mu(s) = \sup_{\norm{g}_{\ell^\infty(V)}\le 1} \sum_{s\in V} f(s)(g(s)-g_V)\mu(s), 
    \end{equation*}
    where, similarly to the first case, H\"{o}lder's inequality and the fact that $f$ sums zero with respect to $\mu$ on $V$ imply that the two sums above are equal. Observe that $g-g_V$ sums zero with respect to $\mu$ on $V$ as in the previous case. Thus, according to Theorem \ref{thm: decomp of f}, we can decompose $g-g_V$ following (\ref{eq: def of decomp}) and writing $g_t$ for simplicity. Hence, we have
    \begin{equation*}
        \norm{f}_{\ell^1(V,\mu)} = \sup_{\norm{g}_{\ell^\infty(V)}\le 1} \sum_{s\in V} f(s) \left(\sum_{t\in V^*} g_t(s)\right)\mu(s) = \sup_{\norm{g}_{\ell^\infty(V)}\le 1} \sum_{s\in V} \sum_{t\in V^*} f(s)g_t(s)\mu(s).
    \end{equation*}
    Next, since the double sum converges absolutely (as in the case $p>1$), we may interchange the order of summation. That is,
    \begin{align*}
        \norm{f}_{\ell^1(V,\mu)} &= \sup_{\norm{g}_{\ell^\infty(V)}\le 1} \sum_{t\in V^*} \sum_{s\in V} f(s)g_t(s)\mu(s)\\
        &= \sup_{\norm{g}_{\ell^\infty(V)}\le 1} \sum_{t\in V^*} \sum_{\substack{s : s=t \text{ or }\\ s=t_p}} f(s)g_t(s)\mu(s) \\
        &= \sup_{\norm{g}_{\ell^\infty(V)}\le 1} \sum_{t\in V^*} \sum_{\substack{s : s=t \text{ or }\\ s=t_p}} (f(s)-f_{\set{t,t_p}})g_t(s)\mu(s) \\
        &\le \sup_{\norm{g}_{\ell^\infty(V)}\le 1} \sum_{t\in V^*} \norm{f-f_{\set{t,t_p}}}_{\ell^1(\set{t,t_p},\mu)}\norm{g_t}_{\ell^\infty(\set{t,t_p})},
    \end{align*}
    where the third line holds since $g_t$ sums zero with respect to $\mu$ on $V$ and thus on the segment $\set{t,t_p}$. We now invoke Lemma \ref{lem: local poincare} to continue this string of inequalities by writing
    \begin{align*}
        \norm{f}_{\ell^1(V,\mu)} &\leq \sup_{\norm{g}_{\ell^\infty(V)}\le 1} \sum_{t\in V^*} \norm{|\ddiff_t f |}_{\ell^1(\set{t,t_p},\mu)}\norm{g_t}_{\ell^\infty(V)} \\
        &\le \sup_{\norm{g}_{\ell^\infty(V)}\le 1} \left(\sum_{t\in V^*} \norm{|\ddiff_t f |}_{\ell^1(\set{t,t_p},\mu)}\right)\left(\sup_{t\in V^*} \norm{g_t}_{\ell^\infty(V)}\right).
    \end{align*}
    Before we proceed, we must estimate $\sup_{t\in V^*} \norm{g_t}_{\ell^\infty(V)}$, where $g_t$ arises from the decomposition
of $g-g_V$. We note that $|g_t(s)|\le c|T(g-g_V)(s)|$ for all $s\in V$. Hence,
    \begin{equation*}
        \norm{g_t}_{\ell^\infty(V)}\le c\norm{T(g-g_V)}_{\ell^\infty(V)}\le c\norm{g-g_V}_{\ell^\infty(V)} \le c(\norm{g}_{\ell^\infty(V)} + \norm{g_V}_{\ell^\infty(V)}),
    \end{equation*}
    where the second inequality follows from Theorem \ref{thm: Hardy strong (infty,infty) cont}, the strong $(\infty, \infty)$ boundedness of the Hardy operator $T$. Furthermore,
    \begin{equation*}
        \norm{g_t}_{\ell^\infty(V)}\le c\left(\norm{g}_{\ell^\infty(V)} + \frac{1}{\mu(V)}\sum_{s\in V}g(s)\mu(s)\right) \le c(2\norm{g}_{\ell^\infty(V)}).
    \end{equation*}
    This holds for all $t\in V^*$, so we adjust our estimate of $\norm{f}_{\ell^1(V,\mu)}$ as follows
    \begin{align*}
        \norm{f}_{\ell^1(V,\mu)} &\le c\sup_{\norm{g}_{\ell^\infty(V)}\le 1} \left(\sum_{t\in V^*} \norm{|\ddiff_t f|}_{\ell^1(\set{t,t_p},\mu)}\right)\left(2\norm{g}_{\ell^\infty(V)}\right)\\
        &\le 2c \sum_{t\in V^*} \left(|\ddiff_t f|(t)\mu(t) + |\ddiff_t f|(t_p)\mu(t_p)\right)\\
        &= 2c \sum_{t\in V^*} \left(|f(t)-f(t_p)|\mu(t) + |f(t_p)-f(t)|\mu(t_p)\right)\\
        &= 2c \sum_{k\in V}\sum_{s : s\sim_T k} |f(s)-f(k)| \mu(k)\\ 
        &\leq 2c \sum_{k\in V}\sum_{s : s\sim k} |f(s)-f(k)| \mu(k)\\ 
        &= 2c\sum_{t\in V} |\ddiff f(t)|\mu(t) = 2c\norm{|\ddiff f|}_{\ell^1(V,\mu)}.
    \end{align*}
    Hence, $C_P\le 2c$ in the special case when $p=1$, and the proof is complete.
\end{proof}

\begin{remark}
Notice that the weighted $k$-ary trees described in Example \ref{ex: k-ary tree is john}, where the weight depends on $\alpha<1/k$, satisfy the condition in Definition \ref{def: John domain} with constant $\frac{1}{1-k\alpha}$. Therefore, we can assert that the global $\ell^p(V,\mu)$-Poincar\'e inequality stated in Theorem \ref{thm: Poincare} is satisfied for $1\le p < \infty$ and the constant can be estimated by  
\begin{equation*}
C_P \le \frac{2\,\left((k+1)\,p\right)^{1-1/p}}{1-k\alpha}.
\end{equation*}
\end{remark}

\begin{remark} Theorem \ref{thm: Poincare} is shown in \cite[Lemma 2.2]{AO}, where the vertices of the graph are the natural numbers and consecutive numbers are adjacent.  This version of the discrete Poincar\'e inequality was used to establish the validity of a certain weighted version of Korn's inequality on bounded domains in $\R^n$ with a boundary singularity.
\end{remark}

\begin{remark} A natural problem arising from this theorem is to determine whether the validity of the Poincar\'e inequality stated in \eqref{eq: global poincare} implies \eqref{eq: John constant}. A related question was studied in the continuous case by S. Buckley and P. Koskela in \cite{BK}, who proved that, under general geometric assumptions on domains in $\R^n$, the validity of the Sobolev–Poincaré inequality implies that the domain is a John domain.
\end{remark}

\section*{Acknowledgements}
The first author gratefully acknowledges Prof. Ojea, the Departamento de Matem\'atica of the Universidad de Buenos Aires, and the Instituto de Investigaciones Matem\'aticas Luis A.~Santal\'o (IMAS, CONICET--UBA) for their warm hospitality during his 2025 sabbatical visit, when this manuscript was partially written. 

In addition, the authors thank the anonymous referee for the careful reading and valuable suggestions that helped improve the presentation of the manuscript, particularly the main result, Theorem \ref{thm: Poincare}.

 \bibliographystyle{plain}
 \bibliography{references.bib}
\end{document}